\documentclass[10pt,a4paper]{article}
\usepackage[margin=1.1cm]{geometry}
\usepackage{amsmath,amsthm,amsfonts,amssymb,amscd,cite,graphicx,array}

\usepackage{titlesec}
\titleformat{\section}
{\normalfont\fontsize{12}{15}\bfseries}{\thesection}{1em.}{}

\newtheorem{corollary}{Corollary}[section]

\newtheorem{theorem}{Theorem}[section]

\allowdisplaybreaks[4]

\let\oldbibliography\thebibliography
\renewcommand{\thebibliography}[1]{%
  \oldbibliography{#1}%
  \setlength{\itemsep}{-2pt}%
}

\usepackage[
  colorlinks   = true, 
  urlcolor     = blue, 
  linkcolor    = blue, 
  citecolor    = blue 
]{hyperref}

\newcommand\oeis[1]{\href{https://oeis.org/#1}{#1}}

\def\ge{\geqslant}
\def\leq{\leqslant}
\def\geq{\geqslant}

\DeclareUnicodeCharacter{03C6}{$\phi$}
\DeclareUnicodeCharacter{03B1}{$\alpha$}
\DeclareUnicodeCharacter{03B2}{$\beta$}

\begin{document}

\baselineskip=0.20in

\makebox[\textwidth]{%
\hglue-15pt
\begin{minipage}{0.6cm}	
\vskip9pt
\end{minipage} \vspace{-\parskip}
\begin{minipage}[t]{6cm}
\footnotesize{ {\bf Discrete Mathematics Letters} \\ \underline{www.dmlett.com}}
\end{minipage}
\hfill
\begin{minipage}[t]{6.5cm}
\normalsize {\it Discrete Math. Lett.}  {\bf X} (202X) XX--XX
\end{minipage}}
\vskip36pt

\noindent
{\large \bf Lattice paths with a first return decomposition constrained by the  maximal height of a pattern}\\

\noindent
Jean-Luc Baril$^{1,}\footnote{Corresponding author (barjl@u-bourgogne.fr)}$,
Sergey Kirgizov$^{1}$\\

\noindent
\footnotesize $^1${\it LIB, Université de Bourgogne Franche-Comté,
  B.P. 47 870, 21078 Dijon Cedex France}\\

\noindent
 (\footnotesize Received: Day Month 202X. Received in revised form: Day Month 202X. Accepted: Day Month 202X. Published online: Day Month 202X.)\\

\setcounter{page}{1} \thispagestyle{empty}

\baselineskip=0.20in

\normalsize

\begin{abstract}
\noindent  We consider the system of equations
  $A_k(x)=p(x)A_{k-1}(x)(q(x)+\sum_{i=0}^k A_i(x))$ for $k\geq r+1$
  where $A_i(x)$, $0\leq i \leq r$, are some given
  functions and show
  how to obtain a close form for $A(x)=\sum_{k\geq 0}A_k(x)$.  We
  apply this general result to the enumeration of certain subsets of
  Dyck, Motzkin, skew Dyck, and skew Motzkin paths, defined
  recursively according to the first return decomposition with a
  monotonically non-increasing  condition relative to the
  maximal ordinate reached by an occurrence of a given pattern $\pi$.
  \\[2mm]
 \end{abstract}

\baselineskip=0.20in

\section{Introduction and notations}
Let $\mathcal{A}$ be a combinatorial class, that is a collection of
similar objects (lattices, trees, permutations, words) endowed with a
size function $\ell$ whose values are non-negative integers, so that
the number $|\ell^{-1}(n)|$ of objects of a given size $n$ is
finite. With $\mathcal{A}$ we associate the generating
function
$$A(x)=\sum\limits_{a\in\mathcal{A}}x^{\ell(a)}=\sum\limits_{n\geq 0}|\ell^{-1}(n)|x^n,$$
which appears to be the main source of
interest for researchers working in enumerative combinatorics.

Given a function $s:\mathcal{A}\rightarrow\mathbb{N}$, called {\it
  statistic}, and an integer $k\geq 0$, we define the combinatorial
class $\mathcal{A}_k$ consisting of objects $a\in\mathcal{A}$ such
that $s(a)=k$, and associate the following generating function with
this class:
$$A_k(x)=\sum\limits_{a\in\mathcal{A}_k}x^{\ell(a)},\mbox{ which implies }A(x)=\sum\limits_{k\geq 0}A_k(x).$$

In three recent articles~\cite{Bar,Bars,flor}, the authors enumerate
particular subclasses of lattice paths (Dyck, Motzkin, $2$-Motzkin and
Schr\"oder paths) defined recursively according to the first return
decomposition with a condition on the path height.  More precisely, in
the first paper, they focus on the Dyck path set
$\mathcal{D}^{h,\geq}$ consisting of the union of the empty path with
all Dyck paths $P$ having a first return decomposition $P=U\alpha
D\beta$ satisfying the conditions:
\begin{equation}
  \begin{cases}
    \alpha,\beta\in \mathcal{D}^{h,\geq},\\
    h(U\alpha D)\geq h(\beta),\\
  \end{cases}
  \label{eq1}
\end{equation}
where $h$ is the {\it height statistic}, i.e. the maximal ordinate
reached by a path.  In these studies, the authors need to consider a
system of equations involving the generating functions $D_k(x)$,
$k\geq 0$, for the subset of paths $P\in \mathcal{D}^{h,\geq}$
satisfying $h(P)=k$:
\begin{equation}
D_k(x)=x\cdot D_{k-1}(x)\cdot\sum\limits_{i=0}^k D_i(x)\mbox{ for } k\geq 1
\label{eq2}
\end{equation}
anchored with initial condition $D_0(x)=1$.

They deduce algebraically a close form of the generating function
$D(x)=\sum_{k\geq 0}D_k(x)$ for $\mathcal{D}^{h,\geq}$ by solving
System~(\ref{eq2}), and prove that $\mathcal{D}^{h,\geq}$ is
enumerated by the Motzkin sequence, that is \oeis{A001006} in the
On-Line Encyclopedia of Integer Sequences (OEIS) \cite{Sloa}. The
second paper is a similar study in the context of Motzkin and
Schr\"oder paths.

Inspired by these two works, the motivation of the paper is to extend
such studies by investigating the solutions of a more general system
defined as follows.  For an integer $r\geq 0$ and two functions $p(x)$
and $q(x)$, we consider the system of equations:

\begin{equation}
A_k(x)=p(x)A_{k-1}(x)\left(q(x)+\sum\limits_{i=0}^k A_i(x)\right)\mbox{, for } k\geq r+1,
\label{eq3}
\end{equation}
anchored with $A_i(x)=u_i(x)$ for $0\leq i \leq r$ where $u_i(x)$,
$0\leq i\leq r$, are some given functions. For the sake of
brevity, we set $u=u_r(x)$, $v=\sum_{i=0}^r u_i(x)-u$, $p=p(x)$ and
$q=q(x)$.

\medskip

Before considering how Equation (\ref{eq3}) is related to particular
sets of lattice paths defined in the same way as the set
$\mathcal{D}^{h,\geq}$, let us recall some definitions and notations.
A {\it skew Motzkin path} (see \cite{lu}) of length $n\geq 0$ is a lattice path in the
quarter plane starting at point $(0,0)$, ending on the $x$-axis (but
never going below), consisting of $n$ steps of the four following
types: up steps $U=(1,1)$, down steps $D=(1,-1)$, flat steps $F=(1,0)$
and left steps $L=(-1,-1)$, such that left and up steps never
overlap. Let $\mathcal{SM}$ be the set of all skew Motzkin paths of
any length. A {\it Motzkin path} is a skew Motzkin path with no left
steps.  Let $\mathcal{M}$ be the set of all Motzkin paths. A {\it Dyck
  path} is a Motzkin path with no flat steps. Let $\mathcal{D}$ be the
set of all Dyck paths. A {\it skew Dyck path} is a skew Motzkin path
with no flat steps. Let $\mathcal{SD}$ be the set of all skew Dyck
paths. Obviously, we have
$\mathcal{D}\subset\mathcal{SD}\subset\mathcal{SM}$ and
$\mathcal{D}\subset\mathcal{M}\subset\mathcal{SM}$.

A {\it pattern} of length $n\geq 1$ in a lattice path $P$ consists of
$n$ consecutive steps. The {\it height} of an {\it occurrence} of a
pattern in a given path is the maximal ordinate reached by its
points. For a pattern $\pi$, we define the statistic $h_\pi$ on
lattice paths so that $h_\pi(P)$ is the maximal height reached by the
occurrences of $\pi$ in $P$. The {\it amplitude} $r_\pi$ of a pattern
$\pi$ is the height of the path $\pi$ considered as a path touching
the $x$-axis in the quarter plane. For instance, the path
$P=UUDUUDUDDD$ contains the pattern $UUD$ and $h_{UUD}(P)=3$, while
the amplitude of $UUD$ is $r_{UUD}=2$.

\medskip

The paper is organized as follows. In Section 2, we show how we can
obtain a close form for $A(x)$ defined by the system (\ref{eq3}). In
Section 3, we provide several concrete applications in the context of
lattices paths (Dyck, Motzkin, skew Dyck, and skew Motzkin paths) for
the statistic $h_\pi$ giving the maximal ordinate reached by the
occurrence of $\pi$ in a path. Finally in Section 4, we provide some
possible future research directions.

\section{General result}

For $k\geq 0$, we set $B_k(x)=\sum_{i=0}^kA_i(x)$, and thus we have
$A(x)=\lim_{k\rightarrow\infty} B_k(x)$.  The
next theorem provides a close form of $A(x)$ whenever
Equation~(\ref{eq3}) admits a solution, that is when $1-A_k(x)p(x)\neq
0$ for all $k\geq r$.

\begin{theorem}
  \label{th1}
  If Equation (\ref{eq3}) has a solution then for $k\geq r+1$ we
  have \begin{equation}B_k(x)=\frac{a+bB_{k-1}(x)}{c+dB_{k-1}(x)},\label{eq4}\end{equation}
  where
$$\begin{cases}
a=p^2qv\left( q+u+v \right) -pqu-u-v\\
b=-p \left( {\it pq}+1 \right)  \left( q+v+u \right) \\
c=-p^2v\left( q+u+v\right)-qp-pv-1\\
d={p}^{2} \left( q+v+u \right) .
  \end{cases}$$
Moreover $A(x)=\lim\limits_{k\rightarrow\infty}
B_k(x)=\sum_{k=0}^\infty A_i(x)$ is a root of the
polynomial
\begin{equation} d\cdot A(x)^2+(c-b)\cdot
  A(x)-a.\label{eq5}
\end{equation}
\end{theorem}
\begin{proof} Using $B_r(x)=u+v$, we solve the following linear system for $a,b,c,d$ 
$$\begin{cases}
B_{r+1}(x)(c+d B_r(x))  = a+bB_r(x)\\
B_{r+2}(x)(c+d B_{r+1}(x))  = a+bB_{r+1}(x)\\
B_{r+3}(x)(c+d B_{r+2}(x))  = a+bB_{r+2}(x),\\
\end{cases}
  $$
and obtain (modulo a multiplicative factor) the claimed values.
However, it remains to prove
that $$B_k(x)=\frac{a+bB_{k-1}(x)}{c+dB_{k-1}(x)}$$ for any $k\geq
r+1$.  We proceed by induction on $k$. Obviously, for $k=r+1$ it is
true. So, we assume that the formula is true for $k\leq n$ and we
prove that this is also true for $k=n+1$. More precisely, we will
prove that $R=(c+dB_n(x))B_{n+1}(x)-(a+bB_n(x))$ equals zero.
 
Considering Equation (\ref{eq3}), we have
$B_{n+1}(x)=\frac{pqA_n(x)+B_n(x)}{1-pA_n(x)}$. The recurrence
hypothesis provides $B_{n}(x)=\frac{a+bB_{n-1}(x)}{c+dB_{n-1}(x)}$,
and we have $A_n(x)=B_n(x)-B_{n-1}(x)$. Replacing respectively
$B_{n+1}(x)$, $B_n(x)$ and $A_n(x)$ in $R$, and multiplying it by $
1-p\cdot A_n(x)\neq 0$, we obtain
{\small$$
 \left( qp \left( {\frac {a+bS}{c+dS}}-S \right) +{\frac {a+bS}{c+dS}}
 \right)  \left( c+{\frac {d \left( a+bS \right) }{c+dS}} \right) -
 \left( a+{\frac {b \left( a+bS \right) }{c+dS}} \right)  \left( 1-p
 \left( {\frac {a+bS}{c+dS}}-S \right)  \right),
 $$}
where $S=B_{n-1}(x)$.
A simplification of this expression (whatever the value of $S$)
implies that $R=0$, which completes the induction.  Finally, taking
the limit when $k\rightarrow\infty$ of both sides of
$B_k(x)=\frac{a+bB_{k-1}(x)}{c+dB_{k-1}(x)}$, and using the fact that
$A(x)=\lim\limits_{k\rightarrow\infty} B_k(x)$, one gets
$$d\cdot A(x)^2+(c-b)\cdot A(x)-a=0.$$
Notice that an alternative proof can be obtained by using Chebyshev polynomials of the second kind (see Proposition D.5 p. 508 in \cite{mansou}).
\end{proof}


\section{Some applications}
In this section, we apply Theorem~\ref{th1} to study special kinds of
classes of lattices paths: Dyck, Motzkin, skew Dyck, and skew Motzkin
paths. We extend the definition~(\ref{eq1}) to these paths using the
maximal height reached by the occurrences of $\pi$, the statistic
$h_\pi$ instead of statistic $h$.  Notice that the statistic $h$
equals the statistic $h_U$ on all considered paths, which allows us to
focus only on $h_\pi$.  We will deal with patterns $\pi$ of length at most three.

Given a pattern $\pi$, we define the set $\mathcal{A}^{h_\pi,\geq}$ as
the union of the empty path with paths $P$ satisfying a recurrence
condition ($R$), which depends on the first return decomposition of a
path. The four cases are described in the following table with
$\alpha,\beta,\gamma \in \mathcal{A}^{h_\pi,\geq}$.

\begin{center}
\begin{tabular}{r|l}
  Paths & First return decomposition  with recurrence condition ($R$)\\[0.5pt]\hline
Dyck &\rule{0pt}{1.6em}%
$P=U\alpha D\beta$ \hspace{1.4em} with $h_\pi(U\alpha D)\geq h_\pi(\beta)$  \\[1ex]
Motzkin &
$\begin{cases}
          P=U\alpha D\beta \\
          P=F\gamma \\
\end{cases}$
          with $h_\pi(U\alpha D)\geq h_\pi(\beta)$ and $h_\pi(F)\geq h_\pi(\gamma)$\\[3.5ex]
Skew Dyck &
$\begin{cases}
          P=U\alpha D\beta \\
          P=U\alpha L \\
\end{cases}$
          $\mbox{with }\begin{cases} h_\pi(U\alpha D)\geq h_\pi(\beta) \\
          \mbox{ and } \alpha\neq \epsilon \mbox{ for the last case}\\  
          \end{cases}$
          \\[3.5ex]
Skew Motzkin &
$\begin{cases}
          P=U\alpha D\beta \\
          P=F\gamma \\
          P=U\alpha L \\
          P=U\alpha L F \gamma\\
\end{cases}$
          $\mbox{with }\begin{cases} h_\pi(U\alpha D)\geq h_\pi(\beta) \mbox{ and }h_\pi(F)\geq h_\pi(\gamma)\\
          \mbox{ and } \alpha\neq \epsilon \mbox{ for the last two cases}\\  
         \end{cases} $\\
\end{tabular}
\end{center}

We keep the notations used in the previous sections. The generating
functions $A_k(x)$, $k\geq 0$, correspond to the sets $\mathcal{A}_k$
of Dyck paths $P\in\mathcal{A}^{h_\pi,\geq}$ so that $h_\pi(P)=k$.  We set
$A(x)=\sum_{k\geq 0}A_k(x)$ and
$B_k(x)=\sum_{i=0}^kA_i(x)$.

In any case, the set $\mathcal{A}_0$ of paths
$P\in\mathcal{A}^{h_\pi,\geq}$ such that $h_\pi(P)=0$ is the set of
paths having no occurrence of $\pi$ at height at least one. In the
case where the amplitude of $\pi$ is at least one, $\mathcal{A}_0$
consists of the paths avoiding $\pi$. So, we can easily obtain
$A_0(x)$ using the first return decomposition of such paths. For Dyck
and Motzkin paths, these generating functions (for short g.f.) are
already known (see \cite{Bre,Deu, mansour, Man1, mansouri, Mer}).  For skew Dyck and
Motzkin paths, we can  obtain them by similar methods, which are
classic and present no difficulty. So, we do not give here all the
details to get $A_0(x)$.

Note that the maximal height of an occurrence of pattern $\pi$ in $P$
is necessarily greater or equal than its amplitude $r_\pi$. This means
that the generating functions for the sets $\mathcal{A}_k$, $1\leq k<
r_\pi$ satisfy $A_k(x)=0$. Finally, if $r_\pi\geq 1$ (which is
equivalent to $\pi\neq F^k$ for all $k\geq 1$) then
$\mathcal{A}_{r_\pi}$, is the set of paths $P\in
\mathcal{A}^{h_\pi,\geq}$ having at least one occurrence of $\pi$ and
such that $h_\pi(P)=r_\pi$, i.e. all occurrences of $\pi$ touch the
$x$-axis in $P$. Similar methods as for $A_0(x)$ can be used in order
to obtain $A_{r_\pi}(x)$. Theorems \ref{th3} and \ref{th4} focus on
two examples of these methods in the context of Dyck paths.

\subsection{Dyck and Motzkin paths}

First we consider Dyck paths. For a given pattern $\pi$, Condition
($R$) implies that for $k\geq r_\pi+1$ we have $$A_k(x)=x\cdot
A_{k-1}(x)\cdot\sum_{i=0}^k A_i(x)$$ which is a particular case of
Equation~(\ref{eq3}) for $r=r_\pi$, $p(x)=x$, $q(x)=0$,
$u=A_{r_\pi}(x)$, and $v=\sum_{i=0}^rA_i(x)-u=A_0(x)$.  From Theorem
\ref{th1}, we can easily deduce the following.
\begin{corollary}
\begin{equation}A(x)={\frac {{x}^{2}uv+{x}^{2}{v}^{2}-ux+1-
\sqrt {\Delta}}{2{x}^{
2} \left( v+u \right) }}\label{eq6}\end{equation}
with $\Delta={u}^{2}{v}^{2}{x}^{4
}+2\,u{v}^{3}{x}^{4}+{v}^{4}{x}^{4}-2\,{u}^{2}v{x}^{3}-2\,u{v}^{2}{x}^
{3}-3\,{u}^{2}{x}^{2}-6\,{x}^{2}uv-2\,{x}^{2}{v}^{2}-2\,ux+1$.
\end{corollary}

In order to convince the reader that all generating functions $A(x)$
can be obtained easily by calculating algebraically $u$ and $v$, we
give here the methods for the two patterns $UUD$ and $DUU$.

\begin{theorem} If $\pi=UUD$ then we have $r_\pi=2$, $u=A_2(x)={\frac {{x}^{2}}{ \left( x-1 \right)  \left( {x}^{2}+x-1 \right) }}$, $A_1(x)=0$, and $v=A_0(x)=\frac{1}{1-x}$.
\label{th3}
\end{theorem}

\begin{proof}
  As already mentioned above, the set $\mathcal{A}_0$ is the set of
  Dyck paths avoiding the pattern $\pi$. So, a nonempty path $P\in
  \mathcal{A}_0$ is of the form $P=UD\alpha$ where
  $\alpha\in\mathcal{A}_0$, which implies that
  $A_0(x)=\frac{1}{1-x}$. On the other hand, the set $\mathcal{A}_1$
  is the set of Dyck paths $P\in\mathcal{A}^{h_\pi,\geq}$ such that
  $h_{UUD}(P)=1$, which cannot be possible because the amplitude
  $r_\pi=2$; so, $A_1(x)=0$. Finally, the set $\mathcal{A}_2$ is the
  set of Dyck paths $P\in\mathcal{A}^{h_\pi,\geq}$ such that
  $h_{UUD}(P)=2$, which means that $P$ can be written $U \alpha D
  \beta$ with $\alpha\in \mathcal{A}_0\backslash \epsilon$ and
  $\beta\in \mathcal{A}_2\cup\mathcal{A}_1\cup\mathcal{A}_0$. So, we
  have the functional equation $A_2(x)=x
  (A_0(x)-1)(A_2(x)+A_1(x)+A_0(x))$, which gives the expected result.
\end{proof}

\begin{theorem} If $\pi=DUU$ then we have $r_\pi=2$, $u=A_2(x)={\frac {{x}^{3} \left(1- x \right) }{ \left( 2\,x-1 \right) ^{2}}}$, $A_1(x)=0$, and $v=A_0(x)=\frac{x-1}{2x-1}$.
\label{th4}
\end{theorem}
\begin{proof}
The set $\mathcal{A}_0$ is the set of Dyck paths avoiding the pattern
$DUU$. So, a nonempty path $P\in \mathcal{A}_0$ is of the form
$P=U\alpha D\beta$ where $\alpha\in\mathcal{A}_0$ and $\beta=(UD)^k$
for some $k\geq 0$, which implies that $A_0(x)=1+\frac{x
  A_0(x)}{1-x}$, and thus $A_0(x)=\frac{x-1}{2x-1}$. On the other
hand, the set $\mathcal{A}_1$ is the set of Dyck paths
$P\in\mathcal{A}^{h_\pi,\geq}$ such that $h_{DUU}(P)=1$, which cannot
be possible because the amplitude $r_\pi=2$; so, $A_1(x)=0$. Finally,
the set $\mathcal{A}_2$ is the set of Dyck paths
$P\in\mathcal{A}^{h_\pi,\geq}$ such that $h_{DUU}(P)=2$, which means
that $P$ can be written $U \alpha D U\beta D$ with $\alpha\in
\mathcal{A}_0$ and $\beta\in \mathcal{A}_0\backslash\{\epsilon\}$.

So,
we have the functional equation $A_2(x)=x^2 A_0(x)(A_0(x)-1)$, which
gives the expected result.
\end{proof}

So, we present in the following table the first values of the
sequences corresponding to the cardinality of
$\mathcal{A}^{h_\pi,\geq}$ for patterns $\pi$ of length at most three.

\begin{center}\begin{tabular}{>{\centering\arraybackslash} m{8em}|c|c}
    Statistic $h_\pi$ & $a_n,~1\leq n \leq 9$ & OEIS \\
    \hline
$h$, $U$, $D$, $UD$, $UU$, $DD$, $UDD$, $UUD$ & 1, 2, 4, 9, 21, 51, 127, 323 & \oeis{A001006}  \\
    $DU$&  1, 2, 4, 8, 17, 39, 94, 233, 588&\\
    $UUU$, $DDD$&1, 2, 5, 13, 35, 97, 274, 786, 2282&  \\
    $UDU$, $DUD$&1, 2, 4, 9, 22, 56, 146, 389, 1053& \\
    $DUU$, $DDU$&1, 2, 5, 13, 34, 89, 234, 621, 1669& \\
\end{tabular}
\end{center}

Now, we examine the set of Motzkin paths. For a
given pattern $\pi$, Condition ($R$) implies that for $k\geq r_\pi+1$
we have $$A_k(x)=x^2\cdot A_{k-1}(x)\cdot\sum\limits_{i=0}^k A_i(x)$$
which is a particular case of Equation~(\ref{eq3}) for $r=r_\pi$,
$p(x)=x^2$, $q(x)=0$, $u=A_{r_\pi}(x)$, and $v=A_0(x)$. Then, the
generating function $A(x)$ is the generating function found above in
the context of Dyck paths (see ($6$)) applied on the variable $x^2$ ($x$ is not replaced with $x^2$ in $u$ and $v$).
The following table gives the first values of generating functions $A(x)$ obtained for all patterns of length at most two.

\begin{center}\begin{tabular}{c|c|c }
Statistic $h_\pi$ & $a_n,~1\leq n \leq 9$ & OEIS  \\ \hline
$h$, $U$, $D$ & 1, 2, 3, 6, 11, 22, 43, 87, 176 & \oeis{A026418}\\
$F$ & 1, 2, 4, 8, 17, 36, 78, 170, 374 & \\ 
$UU$, $DD$&1, 2, 4, 9, 20, 46, 107, 253, 604& \\
$UD$&1, 2, 3, 7, 13, 29, 61, 138, 308& \\
$DU$& 1, 2, 4, 9, 20, 46, 107, 252, 599 & \\
$UF$, $FD$& 1, 2, 4, 8, 17, 37, 82, 185, 422 & \\
$DF$, $FU$ &1, 2, 4, 8, 17, 36, 79, 175, 395 & \\
$FF$&1, 2, 4, 9, 20, 47, 111, 268, 653 & \\
\end{tabular}
\end{center}

Looking at the two tables above, one observes that some patterns have
same sequences, for example $DUU$ and $DDU$, $UF$ and $FD$.  We
actually can generalize this as follows.  Let $\pi = \pi_1 \pi_2
\ldots \pi_k$ be a pattern of length $k$ and $C (\pi)$ be the {\em
  reversed complement} of $\pi$, that is $C (\pi) = \overline{\pi}_k
\overline{\pi}_{k-1} \cdots \overline{\pi}_1 $ where
$$ \overline{x} = 
\begin{cases}
F\, \text{ if } x = F, \\
D\, \text{ if } x = U, \\
U\, \text{ if } x = D. \\
\end{cases}
$$
 For instance we have $C (UFDD) = UUFD$.

\begin{theorem}
If $_\pi A(x)$ is the generating function for the set $\mathcal{A}^{h_\pi,\geq}$, then we have $$ _\pi A(x) = {}_{C(\pi)} A (x).$$
\label{th2}
\end{theorem}

\begin{proof}
For short we set $\sigma=C(\pi)$. Since the reversed complement
operation preserves the amplitude, we have $r_\pi = r_\sigma$. Now, we
will exhibit a bijection $\phi$ from $\mathcal{A}^{h_\pi,\geq}$ to
$\mathcal{A}^{h_\sigma,\geq}$.  Suppose $\pi \ne F^k, k \ge 1$, if a
path $P\in \mathcal{A}^{h_\pi,\geq}$ does not contain $\pi$, we set
$\phi(P)=C(P)$. It avoids $\sigma$ and belongs to
$\mathcal{A}^{h_\sigma,\geq}$. So, we have $_\pi A_0(x) = {}_\sigma
A_0(x)$. If $\pi = F^k$ for some $k\ge 1$, we have $r_\pi = 0$ and we
also set $\phi(P) = C(P)$. We easily deduce again that $_\pi A_0(x) =
{}_\sigma A_0(x)$.

Now, let us consider the case where $r_\pi \ge 1$ and prove that
g.f. $_\pi A_{r_\pi}$ for pattern $\pi$ is actually equals $_\sigma A_{r_\sigma}$
for pattern $\sigma$.  To do this we consider a path $P\in
\mathcal{A}^{h_\pi,\geq}$ having at least one occurrence of $\pi$ and
such that $h_\pi(P)=r_\pi$.  We can write $P$ either as ($i$) $P=F
\alpha$ or ($ii$) $P=U \alpha D \beta$ where $\alpha,\beta\in
\mathcal{A}^{h_\pi,\geq}$.

Case ($i$). We assume $P=F\alpha$ and $h_\pi(F)\geq
h_\pi(\alpha)$. Since the pattern $\pi$ is not $F^k$ (because
$r_\pi\geq 1$), we necessarily have $h_\pi(F)=0=h_\pi(\alpha)$ which
means that $\alpha$ avoids $\pi$. Therefore, the pattern $\pi$ occurs
once in $P$ and it straddles $F$ and $\alpha$. We deduce that $C(P)$
belongs to $\mathcal{A}^{h_\sigma,\geq}$ and it has at least one
occurrence of $\sigma$ such that $h_\sigma(C(P))=r_\sigma$. So in this
case, we define the function $\phi$ with $\phi(P)=C(P)$.

Case ($ii$). We assume $P=U \alpha D \beta$ and $h_\pi(U\alpha D)\geq
h_\pi(\beta)$. If $h_\pi(U\alpha D)=0$ then we have $h_\pi(\beta)=0$
and using the same argument as previous, $C(P)$ belongs to
$\mathcal{A}^{h_\sigma,\geq}$ and it has exactly one occurrence of
$\sigma$ straddling $U\alpha D$ and $\beta$, and such that
$h_\sigma(C(P))=r_\sigma$. In this subcase, we set $\phi(P)=C(P)$.  If
$h_\pi(U\alpha D)=r_\pi$ then the occurrence of $\pi$ in $U\alpha D$
is at the beginning or at the end (or touches both extremities) of
$U\alpha D$ since it must touch the $x$-axis, but the occurrence does not
straddle $U\alpha D$ and $\beta$. This means that $\alpha$ avoids
$\pi$ and $\beta$ either avoids $\pi$ (i.e. $h_\pi(\beta)=0$, as $\pi
\ne F^k$) or contains $\pi$ with $h_\pi(\beta)=r_\pi$.  In this
subcase, we recursively define $\phi$ as $\phi(P)=U C(\alpha) D \phi(\beta)$. It belongs to $\mathcal{A}^{h_\sigma,\geq}$ and has at
least one occurrence of $\sigma$ such that
$h_\sigma(\phi(P))=r_\sigma$. So, $\phi$ is a one-to-one
correspondence between paths in $\mathcal{A}^{h_\pi,\geq}$ having at
least one occurrence of $\pi$ such that $h_\pi(\phi(P))=r_\pi$ and
paths $\mathcal{A}^{h_\sigma,\geq}$ having at least one occurrence of
$\sigma$ such that $h_\sigma(\phi(P))=r_\sigma$. So, we have $_\pi
A_{r_\pi}(x) = {}_\sigma A_{r_\sigma}(x)$.

Since $_\pi A_{k}(x) = {}_\sigma A_{k}(x) = 0$ for $0 < k < r_\pi$,
finally we have $_\pi A_{k}(x) = {}_\sigma A_{k}(x)$ for $0 \leq k
\leq r_\pi$, and Equation~(\ref{eq3}) implies $_\pi A_{k}(x) =
     {}_\sigma A_{k}(x)$ for $k\geq r_\pi+1$.
\end{proof}

\subsection{Skew Dyck  and  Motzkin paths}
Let us discuss the skew Dyck paths first. For a given pattern $\pi$,
Condition ($R$) implies that for $k\geq r_\pi+1$ we
have $$A_k(x)=x\cdot A_{k-1}(x)\cdot\left(\sum\limits_{i=0}^k
A_i(x)+1\right)$$ which is a particular case of Equation~(\ref{eq3})
for $r=r_\pi$, $p(x)=x$, $q(x)=1$, $u=A_{r_\pi}(x)$, and
$v=A_0(x)$. Then, we deduce the following.

\begin{corollary}
\begin{equation}A(x)=\frac {uv{x}^{2}+{v}^{2}{x}^{2}-{x}^{2}u-xu-{x}^{2}+1-\sqrt{\Delta}}{2{x}^{2} \left( 1+v+u \right) }
\label{eq7}
\end{equation}
with $\Delta={u}^{2}{v}^{2}{x}^{4}+2\,u{v}^{3}{x}^{4}+{v}^{4}{x}^{4}+2\,{u}^{2}v{x}^{4}+6\,u{v}^{2}{x}^{4}+4\,{v}^{3}{x}^{4}-2\,{u}^{2}v{x}^{3}+{u}^{2}{x}^{
4}-2\,u{v}^{2}{x}^{3}+6\,uv{x}^{4}+6\,{v}^{2}{x}^{4}-2\,{x}^{3}{u}^{2}
-4\,uv{x}^{3}+2\,u{x}^{4}+4\,v{x}^{4}-3\,{u}^{2}{x}^{2}-6\,uv{x}^{2}-2
\,u{x}^{3}-2\,{v}^{2}{x}^{2}+{x}^{4}-6\,{x}^{2}u-4\,v{x}^{2}-2\,xu-2\,
{x}^{2}+1.$
\end{corollary}

The following table gives the first values of generating functions
$A(x)$ obtained for all patterns of length at most two. Notice that
none sequences appear in \cite{Sloa}.

\begin{center}
    \begin{tabular}{c|c|c }
Statistic $h_\pi$& $a_n,~1\leq n \leq 9$ & OEIS  \\ \hline
$h$, $U$, $D$, $UU$ , $UD$& 1, 3, 8, 23, 68, 211, 668, 2169, 7145&\\
$L$, $DL$&1, 3, 9, 28, 91, 307, 1062, 3748, 13429& \\ 
$DD$ & 1, 3, 9, 29, 96, 327, 1136, 4014, 14365 & \\
$DU$ & 1, 3, 9, 27, 82, 255, 813, 2655, 8847 & \\
$LD$ & 1, 3, 10, 35, 126, 463, 1728, 6529, 24916 & \\
$LL$ &1, 3, 10, 35, 128, 485, 1890, 7531, 30545 & \\
\end{tabular}
\end{center}

In the context of skew Motzin paths we have
$$A_k(x)=x^2\cdot A_{k-1}(x)\cdot\left(\sum\limits_{i=0}^k A_i(x)+x \cdot A_0(x)+1\right)$$
which is a particular case of Equation~(\ref{eq3}) for $r=r_\pi$,
$p(x)=x^2$, $q(x)=x\cdot A_0(x)+1$, $u=A_{r_\pi}(x)$, which means that
$v=0$ when $r_\pi=0$ and $v=A_0(x)$ otherwise. Then, the generating
function $A(x)$ is the generating function for skew Dyck paths (see
Equation (\ref{eq7})) applied on the variable $x^2$ ($x$ is not
replaced with $x^2$ in $u$ and $v$), and we obtain the following
results for patterns of length at most one.  The sequences do not
appear in \cite{Sloa}.

\begin{center}\begin{tabular}{c|c|c}
Statistic $h_\pi$& $a_n,~1\leq n \leq 9$ & OEIS \\ \hline
$h$, $U$ & 1, 2, 4, 9, 20, 45, 101, 229, 524, 1211, 2820 & \\
$D$ &1, 2, 4, 10, 23, 55, 131, 318, 774, 1899, 4678 & \\
$F$ &1, 2, 5, 11, 27, 64, 157, 383, 946, 2347, 5854 & \\
$L$ &1, 2, 5, 12, 30, 76, 196, 513, 1359, 3639, 9831& \\
\end{tabular}
\end{center} 

\section{Conclusion and research directions}

Except Dyck and Motzkin cases, where we obtain two sequences appearing
in \cite{Sloa} for the pattern $U$, all the others have never been
studied in the literature. So, this work provides a new catalog of
sequences which is a possibly fertile ground in number theory, and
more precisely in the analyse of their modular congruences (see
\cite{Desa} for instance). Another research direction could be the
investigation of the distribution of a given pattern in the set
$\mathcal{A}^{h_\pi,\geq}$. For instance, if $\pi=UU$ in the context
of Dyck paths, then the paths in $\mathcal{A}^{h_\pi,\geq}$ avoiding
the pattern $UDU$ are enumerated by the generalized Catalan numbers
(see \oeis{A004148} in \cite{Sloa}), which also count peak-less
Motzkin paths. It would be interesting to find a one-to-one
correspondence between these objects.

In this paper, we focus on the generating $A(x)$ of
$\mathcal{A}^{h_\pi,\geq}$, but we do not deal with the generating
functions $A_k(x)$ for paths $P\in\mathcal{A}^{h_\pi,\geq}$ such that
$h_\pi(P)=k$ since they can be easily obtained. Observing these
g.f. for small $k$, we have identified some of them that are already
known in OEIS \cite{Sloa}. For instance, the number of $n$-length Dyck
paths $P\in\mathcal{A}^{h_{UU},\geq}$ such that $h_{UU}(P)=2$ is the
classical number of Fibonacci minus one (see \oeis{A000071} in
\cite{Sloa}). Such a result suggests us to look for new bijections with known classes counted by this sequence.

Finally, it would be interesting to extend this work to other classes
of objects counted by generating functions satisfying System
(\ref{eq3}).

\section*{Acknowledgements}

This work was partly supported by the project ANER ARTICO
financed by Bourgogne-Franche-Comté region.

\end{document}